\documentclass{amsart}
\usepackage{amsmath,amsfonts,amsthm,amssymb,indentfirst,epic,xurl,graphics,needspace}

\newtheorem{lemma}[equation]{Lemma}

\newtheorem{proposition}[equation]{Proposition}
\newtheorem{definition}[equation]{Definition}
\newtheorem{question}[equation]{Question}

\numberwithin{equation}{section}

\renewcommand{\r}{\mathrm}
\newcommand{\I}{\mathbf{i}\hspace{.05em}}

\newcommand{\langl}{\begin{picture}(.51em,1em)
\put(.11em,.33em){\rotatebox{60}{\line(1,0){.55em}}}
\put(.11em,.33em){\rotatebox{300}{\line(1,0){.55em}}}
\end{picture}}

\newcommand{\rangl}{\begin{picture}(.5em,1em)
\put(.09em,.33em){\rotatebox{120}{\line(1,0){.55em}}}
\put(.09em,.33em){\rotatebox{240}{\line(1,0){.55em}}}
\end{picture}}

\begin{document}

\begin{center}
\texttt{Comments, corrections,
and related references welcomed, as always!}\\[.5em]
{\TeX}ed \today
\\[.5em]
\vspace{2em}
\end{center}

\title%
[Adjoining universal inverses to free monoids]%
{Adjoining universal inverses to families of\\ elements of free monoids}
\thanks{%
Readable at \url{http://math.berkeley.edu/~gbergman/papers/}.
After publication, any updates, errata, related references,
etc., found will be noted at
\url{http://math.berkeley.edu/~gbergman/papers/abstracts}\,.\\
}


\author{George M.\ Bergman}
\address{Department of Mathematics\\
University of California\\
Berkeley, CA 94720-3840, USA}
\email{gbergman@math.berkeley.edu}

\begin{abstract}
Let $\langl X\rangl$ be the free monoid on a generating set
$X,$ and suppose one adjoins to $\langl X\rangl$ universal
$\!2\!$-sided inverses to a finite set $S$ of its elements.
We note an elementary algorithm which yields a normal form for
elements of the resulting monoid~$M.$

In particular, either $M$ will be the free group on~$X,$
or, more generally, the free product as monoids of the
free group on a subset $X_0\subseteq X$ with the free monoid on
$X\setminus X_0,$ or $M$ will contain $\!1\!$-sided invertible elements
that are not $\!2\!$-sided invertible.

We then show that if $S$ is allowed to be infinite, a similar
normal form exists, though it cannot necessarily be
computed algorithmically.

We note work by others on
the related topic of monoids presented by finite
families of relations of the form~$w=1.$
\end{abstract}
\maketitle

\section{Introduction}\label{S.intro}

It is well-known that in a monoid $M,$ if a product $xy$
is invertible (has a $\!2\!$-sided inverse), the
product $yx$ need not be.
The bicyclic monoid $\langl x,\,y~|~xy=1\rangl$
is a standard example.
That example likewise shows that if $xyxy$ is invertible, none
of $yx,$ $yxy,$ or $yxyx$ need be.

On the other hand, if a product $xyx$ is invertible, we see
that $x$ has both a right and a left inverse, hence it
is invertible, hence so is $y,$ so the submonoid
of $M$ generated by $x,$ $y$ and $(xyx)^{-1}$ is a group.

What is going on here?

In this note, we study the structure of the monoid $M$ obtained by
starting with the free monoid $\langl X\rangl$ on a set $X,$ and
for some subset $S$ of its elements, adjoining universal inverses
to the members of~$S.$
In \S\ref{S.main} we develop a simple algorithm for
obtaining a well-behaved normal form for $M$ if $S$ is finite,
and in \S\ref{S.observations} make some further observations
on this construction.

If $S$ is not assumed finite, $M$ will be the direct
limit of the monoids arising in this way from finite subsets
$S'\subseteq S,$ and in \S\ref{S.inf} we show that $M$
inherits most of the properties obtained in the finite case
-- the notable exception being the property that its structure
can be computed algorithmically from the set~$S.$

In \S\ref{S.special}, we briefly compare these results
with results in the literature on monoids presented using
families of relations of the form~$w=1.$

\section{General conventions and observations}\label{S.easy}

We make here some straightforward
conventions and observations, so that
we will be able to call on them subsequently without digressing.
First, a convention.\vspace{.3em}
\begin{equation}\begin{minipage}[c]{35pc}\label{d.inv}
The words ``inverse'' and ``invertible'', unless modified by
``left'', ``right'' or ``one-sided'', will mean two-sided inverse
and two-sided invertible.\vspace{.35em}
\end{minipage}\end{equation}
Now some trivial observations.
Let $M$ be any monoid.\vspace{.3em}
\begin{equation}\begin{minipage}[c]{35pc}\label{d.rt+l}
If an element $w \in  M$ has at least one right inverse and
at least one left inverse, then $w$ is invertible in $M.$
All one-sided inverses of $w$ in $M$ are then equal,
and their common value is a $\!2\!$-sided inverse of $w.$\vspace{.35em}
\end{minipage}\end{equation}
\begin{equation}\begin{minipage}[c]{35pc}\label{d.div-1-sided_inv}
If an element $w \in M$ is right invertible in $M,$ then all left
divisors of $w$ are also right invertible in $M.$
If an element $w \in M$ is left invertible in $M,$ then all right
divisors of $w$ are left invertible in~$M.$
\end{minipage}\end{equation}\vspace{0em}

On the other hand, a {\em right} divisor of a right invertible
element need not, in general, be right invertible, and analogously
for a left divisor of a left invertible element.
E.g., in the monoid $\langl x,\,y,\,z~|~x\,y\,z=1\rangl,$
$y$ is a right divisor of the right-invertible
element $xy,$ but is not right-invertible,
and is a left divisor of the left-invertible
element $yz$ but not left-invertible.
However, a special case, where those implications are easily seen
to hold is:\vspace{.3em}
\begin{equation}\begin{minipage}[c]{35pc}\label{d.pull_unit_fr_1-sided}
If a right invertible element $w$ has a factorization
$w = uv,$ where $u$ is invertible, then $v$ is also right invertible.
Likewise, if a left invertible element $w$ has a factorization $w = vu,$
and $u$ is invertible, then $v$ is also left invertible.\vspace{.3em}
\end{minipage}\end{equation}
To see the right invertibility case, one conjugates an equation showing
right invertibility,
$uvz=1,$ by $u,$ and analogously for the left invertibility case.

Another notational convention:\vspace{.5em}
\begin{equation}\begin{minipage}[c]{35pc}\label{d.i(w)}
If $w$ is an element of a monoid $M,$ and we universally adjoin
to $M$ an inverse to $w,$ we will denote that inverse
$\I(w)$ and the resulting monoid $M\langl\I(w)\rangl.$
Similarly, if we adjoin inverses to all elements in
a subset $S\subseteq M,$ we shall use the notation
$M\langl\I(S)\rangl,$ or, if
$S=\{w_1,\dots,w_n\},$ $M\langl\I(w_1),\dots,\I(w_n)\rangl.$
\end{minipage}\end{equation}

(It would feel more natural to write $w^{-1}$ rather than $\I(w);$
but the notation $w^{-1}$ applies to all cases where an element $w$ has
an inverse, and I want a notation specific to an inverse universally
adjoined to an element not assumed to have one.)
\begin{equation}\begin{minipage}[c]{35pc}\label{d.same_symb}
In the context of~\eqref{d.i(w)} we will
generally use the same symbols for elements
of $M$ and for their images in $M\langl\I(S)\rangl$ (even
though elements that are distinct in $M$ may fall together
in $M\langl\I(S)\rangl).$
\end{minipage}\end{equation}

For an example where such elements fall together, let
$M= \langl x,\,y~|~xy=1\rangl.$
Then $yx$ and $1$ are distinct in $M,$ but fall together
in $M\langl\I(x)\rangl.$

Next, some observations requiring a bit more thought.

\begin{lemma}\label{L.S_&_S'}
Let $M$ be a monoid, and $S,$ $S'$ subsets of $M.$
Then the following conditions are equivalent:

{\rm(i)} In $M\langl\I(S)\rangl,$ all elements of $S'$ become
invertible, and
in $M\langl\I(S')\rangl,$ all elements of $S$ become invertible.

{\rm(ii)} There exist homomorphisms
$f: M\langl\I(S)\rangl\to M\langl\I(S')\rangl$
and $g: M\langl\I(S')\rangl\to M\langl\I(S)\rangl$
respecting the natural homomorphisms of $M$ into these monoids.

In this situation, the $f$ and $g$ of {\rm(ii)}
are in fact inverse to one another,
and give the unique isomorphism between
$M\langl\I(S)\rangl$ and $M\langl\I(S')\rangl$
respecting the homomorphisms of $M$ into those monoids.
\end{lemma}

\begin{proof}
(i) implies (ii) by the universal properties
of $M\langl\I(S)\rangl$ and $M\langl\I(S')\rangl.$
The reverse implication is trivial.
The final statement is easily deduced from the uniqueness
of inverses of invertible elements of a monoid.
\end{proof}

\begin{definition}\label{D.inv-eq}
If $M$ is a monoid, we shall call two subsets $S$
and $S'$ of $M$ {\em inverse-equivalent} if they satisfy the equivalent
conditions of Lemma~\ref{L.S_&_S'}.
\end{definition}

The following lemmas describe, in the context of a general monoid $M,$
two tools that we will apply in the next section to free monoids.

\begin{lemma}\label{L.delete_one}
Let $M$ be a monoid and $S$ a subset of $M.$
If $S$ contains two elements $u$ and $v$ and
their product $uv,$ then $S$ is inverse-equivalent to the
subset one gets by deleting from $S$ any {\em one} of these
three elements, as long as it is distinct from both of the others
{\rm(}though those two others
need not be distinct from each other{\rm)}.
\end{lemma}

\begin{proof}
It will suffice to show that for the given $S,$ condition (i) of
Lemma~\ref{L.S_&_S'}
holds for $S'$ the set obtained by deleting from $S$ the
element in question.
Clearly every element of $S'$ will be invertible
in $M\langl\I(S)\rangl;$ it remains to show that the
element we have deleted from $S$ in forming $S'$ remains
invertible in $M\langl\I(S')\rangl.$
If that element is $uv,$ this is true because a product of
invertible elements is invertible.
If it is $u,$ it will again be the product of
invertible elements $u = (uv)\,\I(v),$
and if it is $v,$ it will likewise be the product $v =\I(u)(uv).$
\end{proof}

\begin{lemma}\label{L.left-div_rt-div}
Let $M$ be a monoid, $S$ a subset of $M,$ and $w$ any element of $M.$
Then if, in $M,$ $w$ is a left divisor of some $u\in S$ and
is also a right divisor of some $v\in S,$ then $S\cup{\{w\}}$
is inverse-equivalent to~$S.$
\end{lemma}

\begin{proof}
The one non-obvious condition we need is that the image of $w$
in $M\langl\I(S)\rangl$ is invertible.
By~\eqref{d.div-1-sided_inv}, that image is right and left invertible
in $M\langl\I(S)\rangl,$ so by~\eqref{d.rt+l}, it is invertible there.
\end{proof}

\section{Monoids obtained by adjoining to a
free monoid inverses to finitely many elements}\label{S.main}

We turn, at last, to the case of extensions of free monoids.

\begin{definition}\label{D.main}
For any set $X,$ the free monoid on $X$ will be
denoted $\langl X\rangl.$
The result of adjoining universal inverses to
elements of $\langl X\rangl$ {\rm(}which
in the general context of~\eqref{d.i(w)} would be denoted
$\langl X\rangl\langl\I(w_1),\dots,\I(w_n)\rangl$
or $\langl X\rangl\langl\I(S)\rangl)$ will be
written $\langl X:\,\I(w_1),\dots,\I(w_n)\rangl$ or
$\langl X:\,\I(S)\rangl.$

We shall say that two nonidentity elements of $\langl X\rangl$
{\em overlap} if one of them can be written $ab$ and the other $bc,$
where $b\neq 1,$ and at most one of $a$ and $c$ equals $1.$
{\rm(}Note that if $a = 1,$
then $ab=b,$ a proper left substring of $bc,$ while
if $c = 1,$ $bc=b,$ a proper right substring of $ab.)$
If neither $a$ nor $c$ equals $1,$ we shall say that $ab$ and $bc$
{\em overlap noninclusively}.

We shall call a subset $T\subseteq \langl X\rangl\setminus\{1\}$
{\em overlap-free} if no pair of elements of $T$
{\rm(}distinct or not{\rm)} overlap.
\end{definition}

(The point of the phrase ``distinct or not'' is to make clear
that an overlap-free set is not allowed to contain
self-overlapping elements, such as $x\,y\,x$ or $x\,x.)$

\begin{proposition}\label{P.oE_T}
Let $X$ be a set, and $S$ any finite subset of the free monoid
$\langl X\rangl.$
Then there is a finite overlap-free
subset $T$ of $\langl X\rangl\setminus\{1\}$
inverse-equivalent to $S$ {\rm(}Definition~\ref{D.inv-eq}{\rm)}.

Such a $T$ can be constructed algorithmically from $S.$
\end{proposition}

\begin{proof}
If $1\in S,$ then deleting that element clearly will not change
the inverse-equivalence class of $S,$ so we may assume without loss
of generality that $S\subseteq\langl X\rangl\setminus\{1\}.$

We will define below a function associating to every finite
subset $U$ of $\langl X\rangl\setminus\{1\}$ a nonnegative
integer ``weight'', $\mathrm{wt}(U),$ and show that
if $U$ has two elements which overlap, we can obtain
a set $V$ inverse-equivalent to $U,$ but having lower weight.
Thus, starting with $S,$ this process must terminate in a set $T$
inverse-equivalent to $S$ which is overlap-free.

For any $w\in \langl X\rangl,$ let
$\mathrm{lgth}(w)$ denote its length
as a word in the elements of $X,$ and let
\begin{equation}\begin{minipage}[c]{35pc}\label{d.wt=}
$\mathrm{wt}(U) \ = \ \sum_{w\in U}\,(\r{lgth}(w)-1).$
\end{minipage}\end{equation}

Suppose first that $U$ contains two distinct
elements, one of which left-divides the other: $u$ and $uv.$
Then by Lemma~\ref{L.delete_one},
$U\cup\{v\}$ and $U$ are inverse-equivalent, while by another
application of that Lemma, $U\cup \{v\}$ is inverse-equivalent
to $U\cup\{v\}\setminus\{uv\};$
so $U$ is inverse-equivalent to $U\cup \{v\} \setminus\{uv\}.$
Now adjoining $\{v\}$ increased the weight of $U$
by at most $\mathrm{lgth}(v)-1$ (``at most'' because $v$ might
already have belonged to $U,$ in which case bringing it in
left $U$ unchanged); and deleting $uv$ decreased
that weight by the larger value $\mathrm{lgth}(uv)-1.$
So the combined changes decrease the weight.

Similarly, if $U$ contains two distinct elements one of which
{\em right}-divides the other, then the left-right dual of the
above construction decreases its weight.

Finally, if $U$ contains elements $ab$ and $bc$
(not necessarily distinct) which
overlap noninclusively, then by Lemma~\ref{L.left-div_rt-div},
$U\cup\{b\}$ is inverse-equivalent to $U,$ and after
adjoining $b$ to $U,$ we can then replace $ab$ by $a.$
These steps together show that $U$ is inverse-equivalent to
$(U\setminus\{ab\})\cup\{a,\,b\};$ and again, this decreases
the weight of $U$ by at least~$1.$
(Why was the last step above just to
replace $ab$ by $a,$ and not simultaneously replace $bc$ by $c$?
That works nicely if $ab$ and $bc$ are distinct, but if they are equal,
it requires us to consider various subcases.
To avoid these complications, I chose to just delete
one of these products in this step.
If removing $ab$ still leaves $bc$ in our set, then at a subsequent
step we can handle this as in the above
``one left-divides the other'' case.)

Since the initial value of $\mathrm{wt}(S)$ was a positive
integer, the above process must terminate after
finitely many steps, giving a $T$
inverse-equivalent to $S$ which cannot be
so modified, i.e., which is overlap-free.
\end{proof}

The above construction involves repeated choices of
which elements to add to and remove from $U,$ and so does
not show $T$ to be uniquely determined by $S.$
But we will be able to prove that uniqueness
after getting some further results.

These intermediate results, which will also be used
in \S\ref{S.inf}, will not be restricted to finite sets~$T.$

The next result makes use of a tool for establishing normal forms,
the Diamond Lemma~\cite{<>}.
Though the title of~\cite{<>} refers to ring theory,
\S9.1 thereof notes the version of the result
for monoids, there called ``semigroups''.
(A similar procedure for establishing normal form results,
often used in semigroup and monoid theory, is called
``Knuth-Bendix reduction''~\cite[\S12.2]{H+E+O}.
As discussed at~\cite[p.\,179]{<>}, the term ``Diamond Lemma'' goes
back to a graph-theoretic result of M.\,H.\,A.\,Newman~\cite{Newman}.)

\begin{proposition}\label{P.nml_form}
Let $X$ be a set, and $T$ an {\em overlap-free} subset
of $\langl X\rangl\setminus\{1\}.$
Then every element
of $\langl X:\,\I(T)\rangl$ can be written uniquely as
the product of a {\rm(}possibly empty{\rm)} string of elements
of $X\cup\,\I(T)$ in which no element $\I(w)$ $(w\in T)$ is
immediately preceded or followed by the corresponding string $w.$
\end{proposition}

\begin{proof}
That every element of $\langl X :\,\I(T)\rangl$ can be written as
such a product is immediate even without the assumption that $T$
is overlap-free.
Indeed, by definition of
$\langl X:\,\I(T)\rangl,$ every element thereof is a product
of elements of $X\cup\,\I(T).$
If such a product contains an $\I(w)$
immediately preceded or followed by $w,$
then the product obtained by dropping the substring
$w\,\I(w)$ or $\I(w)\,w$ represents the same element,
and is shorter.
Repeating such rewritings, we must eventually get an expression
of the desired sort.

To show that such representations are unique, we
use the Diamond Lemma~\cite{<>}.
Given $X$ and $T,$ the defining identities for
$\langl X:\,\I(T)\rangl$ yield the reduction system
consisting of the reductions
\begin{equation}\begin{minipage}[c]{35pc}\label{d.redns}
$(w\,\,\I(w),~1)$ \ and \ $(\,\I(w)\,w,~1)$ \ for all \ $w\in T.$
\end{minipage}\end{equation}

Note that each generator of
the form $\I(w)$ $(w\in T)$ appears at the end
of just one of the $2\,\r{card}(T)$
reducible strings shown in~\eqref{d.redns},
and at the beginning of just one such string,
and nowhere else.
Hence the only overlaps among such strings in which
the overlapping substrings involve an $\I(w)$ correspond to the
words $w\,\,\I(w)\,w,$ and it is clear that the two ways
of reducing such a word give the same value, $w.$
So in the language of~\cite{<>}, the ambiguities
whose overlapping portions involve an $\I(w)$ are resolvable.

What about overlaps involving the strings $w$?
By assumption, no two members of $T,$ distinct or
not, overlap; this leaves only
the case where the right-hand $w$ of a product $\I(w)\,w$
is equal to the left-hand $w$ of the product $w\,\I(w),$
and the overlapping subword is in the whole string $w.$
In that case, the ambiguously reducible product is
$\I(w)\,w\,\,\I(w),$ and both reductions yield $\I(w).$

So all ambiguities of our reduction-system are resolvable, hence
the result of~\cite{<>} for monoids shows that each element
of $\langl X:\,\I(T)\rangl$ has a unique reduced expression;
i.e., a unique expression of the asserted form.
\end{proof}

We next want to prove that the {\em invertible} elements of
$\langl X:\,\I(T)\rangl$ are just the obvious ones,
and so, in particular, obtain a simple description
of which elements of the original monoid $\langl X\rangl$
become invertible in \mbox{$\langl X:\,\I(T)\rangl.$}
Curiously, the route to these results requires
us to first determine which elements become
$\!1\!$-sided invertible.

\begin{proposition}\label{P.rt-or-left-inv}
Let $X$ be a set, and $T$ an overlap-free subset
of $\langl X\rangl\setminus\{1\}.$
Then the right-invertible
elements of $\langl X:\,\I(T)\rangl$ are precisely the elements
which, when written in the normal form
of Proposition~\ref{P.nml_form}, consist of products
in any order {\rm(}including the empty product~$1)$
of elements of $\I(T)$ and nonempty {\em initial} substrings
{\rm(}not necessarily proper{\rm)} of the expressions
in $\langle X\rangl$ for elements of $T.$

In particular, the elements
of the free monoid $\langl X\rangl$ that become right-invertible in
$\langl X:\,\I(T)\rangl$ are the products of
nonempty initial substrings of the expressions
in $\langl X\rangl$ for members of $T.$

The elements of $\langl X:\,\I(T)\rangl$ and $\langl X\rangl$ that
are {\em left}-invertible in $\langl X:\,\I(T)\rangl$
are described by the left-right duals of the above characterizations.
\end{proposition}

\begin{proof}
By symmetry it suffices to prove the assertions
on right-invertible elements.

So suppose $u,\,v\in \langl X:\,\I(T)\rangl$ satisfy $uv=1.$
Assume inductively that all right-invertible words
shorter than $u$ (i.e., having shorter normal
forms in terms of the generating set $X\cup\,\I(T))$
are products of the asserted form.

Write $u$ and $v$ in normal form.
If they are not both $1,$ then the string $u\,v$ must
not be in normal form, so the reducibility of $u\,v$ must result
from some nonempty final substring of $u,$ when adjoined
to some initial substring of $v,$ giving an expression $w\,\,\I(w)$
or $\I(w)\,w$ for some $w\in T.$
(If both of these situations occurred
simultaneously, i.e., if when the expressions
for $u$ and $v$ were adjoined, the result contained
a substring $\I(w)\,w\,\,\I(w)$ or $w\,\,\I(w)\,w,$
one could arbitrarily use one of the substrings
$\I(w)\,w$ or $w\,\,\I(w)$ thereof in the argument to follow.
But in fact, as we shall note after the proof, this cannot happen.)

If a substring $w\,\I(w)$ arises in this way, that
means the expressions for $u$ and $v$ have the forms
\begin{equation}\begin{minipage}[c]{35pc}\label{d.u=,v=}
$u\,=\,u'\,w_1,$ \ $v\,=\,w_2\,\I(w)\,v',$ \ where \ $w_1\,w_2\,=\,w.$
\end{minipage}\end{equation}
Here the string $w_1$ has to be nonempty, since otherwise the
expression for $v$ would begin $w\,\I(w),$ contradicting its
being in normal form.
Note also that $u',$ as a left divisor of the right-invertible
element $u,$ is right invertible (by~\eqref{d.div-1-sided_inv}), so
by our inductive hypothesis, $u'$ is a product of the asserted form.
Hence, since $w_1$ is an initial substring of $w\in T,$
we see that $u=u'\,w_1$ is also a product of the asserted form.

On the other hand, if the substring that makes $u\,v$ reducible
has the form $\I(w)\,w,$ then the normal forms for $u$ and $v$
must have the forms $u'\,\I(w)\,w_1$ and $w_2\,v',$ where
again $w_1\,w_2=w,$ making $w_1$ an initial substring
(possibly empty) of a member of $T.$
Again, $u',$ as a left divisor of a right-invertible
element, is right invertible, so
again applying our inductive hypothesis to $u',$
we get the desired conclusion for $u=u'\,\I(w)\,w_1.$

The assertion beginning ``In particular'' clearly follows.
\end{proof}

We remarked parenthetically that the situation where two elements
$u,$ $v$ with $u\,v=1$ have normal forms that, when
put together, contain
a string $\I(w)\,w\,\,\I(w)$ or $w\,\,\I(w)\,w$ cannot actually occur.
Indeed, if $u=u'\,\,\I(w)\,w_1$ and $v=w_2\,\,\I(w)\,v'$
where $w_1\,w_2=w,$ then the
product $u'\,\,\I(w)\,w_1\,w_2\,\,\I(w)\,v'=u'\,\,\I(w)\,w\,\,\I(w)\,v'$
would reduce to $u'\,\I(w)\,v'.$
For this to reduce to $1,$ it would in particular have
to be reducible, so either $u'$ would have to end
in $w,$ or $v'$ would have to begin with $w;$ but either
of these properties would contradict the assumption that
$u$ and $v$ were in normal form.
The case where putting together the normal forms
of $u$ and $v$ gave an expression containing $w\,\,\I(w)\,w$
is still simpler: if the indicated
occurrence of $\I(w)$ came from the expression for $u,$ then that
expression would not have been in normal form, and dually
if it came from the expression for $v.$

We remark that if an element of a
monoid $\langl X:\,\I(T)\rangl$ has a $\!1\!$-sided
inverse, that $\!1\!$-sided inverse need not be unique.
E.g., in $\langl x,\,y,\,z:\,\I(xy),\,\,\I(xz)\rangl,$
the element $x$ has distinct right inverses
with normal forms $y\,\,\I(xy)$ and $z\,\,\I(xz).$

Back to general results.
We can now characterize $\!2\!$-sided invertible elements:

\begin{proposition}\label{P.inv}
Let $X$ be a set, and $T$ an overlap-free subset
of $\langl X\rangl\setminus\{1\}.$
Then the invertible elements
of $\langl X:\,\I(T)\rangl$ are precisely those which can be
written as products of elements of~$T\cup\,\I(T).$
The normal form {\rm(}as in Proposition~\ref{P.nml_form}{\rm)}
of each such element can be written uniquely as such a product
which contains no substring $w\,\I(w)$ or $\I(w)\,w$~$(w\in T).$

In particular, the elements of $\langl X\rangl$ that are invertible in
$\langl X:\,\I(T)\rangl$ are precisely the products of
elements of~$T,$ and each has a unique
expression as such a product.

Thus, the monoid of all elements of $\langl X\rangl$
with inverses in $\langl X:\,\I(T)\rangl$
is the free monoid on $T,$ and the group of all invertible elements
of $\langl X:\,\I(T)\rangl$ is the free group on~$T.$
\end{proposition}

\begin{proof}
We shall first show that every invertible element
has an expression as in
the first paragraph above, then that such expressions are unique.
These results clearly imply the remaining assertions.

Let $u$ be an invertible element of $\langl X:\,\I(T)\rangl.$
If $u=1,$ then $u$ is the empty product of members of~$T\cup\,\I(T),$
so let $u\neq 1,$ and assume inductively that
all invertible elements having shorter normal forms
than that of $u$
are products of elements of~$T\cup\,\I(T).$

Since $u$ is right invertible, Proposition~\ref{P.rt-or-left-inv}
says that it can be written $u=u_1\dots u_n$ where all
$u_i$ are either nonempty left substrings of elements of $T$
(hence right invertible) or members of $\I(T)$ (hence invertible);
and since it is left invertible, that result likewise tells
us that it can be written $u=u'_1\dots u'_{n'}$ where all
$u'_i$ are either right factors of elements of $T$
or members of $\I(T).$

Let us first consider the case where $n=n'=1.$
Then $u$ is either a member of $\I(T),$
or is both a left substring of a member
of $T$ and a right substring of a member of $T.$
The former case clearly satisfies our desired conclusion,
while in the latter, the condition that $T$ be overlap-free
tells us that $u\in T,$ so that conclusion again holds.

It remains to consider the case where $n$ and $n'$ are not both~$1.$
By left-right symmetry, it suffices to assume $n>1,$
so that $u$ is a product $u_1\dots u_n$ of $n>1$ elements
$u_i,$ each of which is either a member of $\I(T),$
or a nontrivial left factor of a member of~$T.$

In this case, note that by assumption, the final
factor $u_n$ is right invertible; but being a right factor
of the invertible element $u,$ it is also left invertible.
Thus, $u_n$ is invertible; hence both of the factors
in the decomposition $u=(u_1\dots u_{n-1})\,u_n$ are
invertible elements of shorter normal form than $u.$
Hence by our inductive assumption, they are both products
of members of~$T\cup\,\I(T);$ hence so is $u,$ as desired.

As for the uniqueness of the decomposition
as a product of members of~$T\cup\,\I(T),$
suppose the normal form for $u$ has two such decompositions,
$u=u_1\dots u_n =u'_1\dots u'_{n'},$ and
assume inductively that for invertible
elements with shorter normal forms, the desired uniqueness holds.
If the normal form for $u$ begins with an element of $\I(T),$
then each of our decompositions must have
that as its first factor, and the result of deleting that factor
gives $u_2\dots u_n =u'_2\dots u'_{n'},$ so
our inductive assumption leads to the desired conclusion.
If the normal form for $u$ does not begin with an element
of $\I(T),$ then the initial substrings $u_1$ and $u'_1$
will both be members of $T,$ one of which is a left substring
of the other.
But since $T$ is assumed non-overlapping,
no member of $T$ can be a proper left substring of another, so
we again get $u_1=u'_1,$ and conclude as before that the
decompositions are the same.
\end{proof}

We can now, at last, get the uniqueness of the $T$
constructed from an arbitrary finite $S$
in Proposition~\ref{P.oE_T}.
Indeed, the above result clearly implies

\begin{proposition}\label{P.T=mnl}
Let $X$ be a set, and $T$ an overlap-free subset
of $\langl X\rangl\setminus\{1\}.$
Then $T$ is precisely the set of nonidentity
elements of $\langl X\rangl$ which are invertible
in $\langl X:\I(T)\rangl,$ but have no
factorization into two nonidentity elements of $\langl X\rangl$
that are invertible in $\langl X:\I(T)\rangl.$

Hence if $S$ is any finite subset of $\langl X\rangl,$
and $T$ is an overlap-free subset inverse-equivalent to $S$
{\rm(}shown to exist by Proposition~\ref{P.oE_T}{\rm)},
then $T$ is uniquely determined by $S{\,:}$ it is the
set of nonidentity elements of $\langl X\rangl$ that are invertible
in $\langl X:\I(S)\rangl$ but have no factorization into two
nonidentity elements of $\langl X\rangl$
invertible in $\langl X:\I(S)\rangl.$\qed
\end{proposition}

Note that although the meaning given to ``overlap'' in
Definition~\ref{D.main} includes the case where one of the two words
is an initial or final subword of the other, it does
not include the case where one is merely an {\em internal}
subword of the other.
For instance, $x\,y\,z$ and $y$ do not count
as overlapping words;
so $\langl x,\,y,\,z:\,\I(xyz),~\I(y)\rangl$ is a presentation
$\langl X:\,\I(T)\rangl$ with overlap-free $T.$
Hence, in this monoid, $z\,\I(xyz)\,x\neq\I(y),$
though these elements fall together if multiplied on
the left by $x\,y,$ or on the right by $y\,z.$

\section{Some observations}\label{S.observations}

What does the monoid $\langl X:\I(T)\rangl$
of Proposition~\ref{P.inv} look like?

Clearly, if $T=X,$ it will be the free group on~$X.$
If, more generally, $T$ is a subset of~$X,$ then
it is easy to see that $\langl X:\I(T)\rangl$
will be the free product as monoids of the free group
on $T$ and the free monoid on $X\setminus T.$

On the other hand, if $T$ is not a subset of $X,$ i.e., if
$T$ contains at least one word $w$ of length $>1,$
then writing $w$ as a product $w=w_1\,w_2$ of nonempty words,
we see that $w_1$ will be right invertible but not
left invertible, with right inverse $w_2\,\I(w).$
Writing $x=w_1$ and $y= w_2\,\I(w),$
these elements will generate a bicyclic monoid
$\langl x,\,y~|~xy=1\rangl,$ since the relation
$xy=1$ allows us to reduce every product of $x$ and $y$
to the form $y^m\,x^n$ $(m,n\geq 0),$ while from
Proposition~\ref{P.nml_form} we see that no two
of these elements are equal.

Such monoids $\langl X:\I(T)\rangl$
will contain nonidentity idempotent elements,
since in a bicyclic monoid $\langl x,\,y~|~xy=1\rangl,$
$y\,x$ is such an element.
(However, I suspect that if $w$ is any element of a monoid
$\langl X:\I(T)\rangl$ which is {\em not} idempotent, then all the
positive integer powers of $w$ are distinct.)

The construction in Proposition~\ref{P.oE_T} of an
overlap-free set $T$ inverse-equivalent to a given finite
set $S$ tends to {\em increase} the number of elements in the set.
On the other hand, if one is not concerned with overlap-freeness,
one can easily {\em decrease} that number to~$1\!:$

\begin{lemma}\label{L.1-elt}
Let $X$ be a set, and $S$ a finite subset of $\langl X\rangl.$
Then there is an element $w\in \langl X\rangl$ such
that the singleton $\{w\}$ is inverse-equivalent to~$S.$
\end{lemma}

\begin{proof}
If $S$ is empty, we can take $w=1,$ and if $S$ is
a singleton $\{w_1\}$ we can take $w=w_1,$ so assume
$S=\{w_1,\,\dots,\,w_n\}$ with $n>1,$ and let
\begin{equation}\begin{minipage}[c]{35pc}\label{d.w=}
$w \ = \ %
w_1\,\dots\,w_{n-1}\,w_n\,w_{n-1}\,\dots\,w_1.$
\end{minipage}\end{equation}
I claim that the technique used in
the last part of the proof of Proposition~\ref{P.oE_T},
applied repeatedly, shows $\{w\}$ inverse-equivalent to~$S.$

Indeed, note that $w$ overlaps itself
noninclusively via the initial and final factors $w_1.$
Hence as in the last part of the proof of Proposition~\ref{P.oE_T}
we can get an inverse-equivalent set by replacing the
singleton $\{w\}=\{w_1\dots w_{n-1} w_n w_{n-1}\dots w_1\}$
with the $\!2\!$-element set
$\{w_1,\ w_2\dots w_{n-1} w_n w_{n-1}\dots w_1\},$
which we can in turn replace with
$\{w_1,\ w_2\dots w_{n-1} w_n w_{n-1}\dots w_2\}.$
Iterating this process, we eventually get $\{w_1,\dots,w_n\}$
i.e., $S.$
(Of course, since we did not assume $S$ overlap-free,
there may be overlaps among the $w_i,$ which in the proof
of Proposition~\ref{P.oE_T} we would have
gotten rid of by further steps.
But our goal here isn't to get an overlap-free $T,$
so we only carry out the steps indicated above.)
\end{proof}

I will mention here a curious observation not directly related to the
subject of this paper.
It seems unlikely that it is not known; if someone points
out a reference, I will delete the result.

\begin{lemma}\label{L.bicyc}
If two elements $x$ and $y$ of a monoid $M$ satisfy $xy=1$
but $yx\neq 1,$ then
\begin{equation}\begin{minipage}[c]{35pc}\label{d.bicyc}
the map $M\to M$ given by $z\mapsto y\,z\,x$
\end{minipage}\end{equation}
is a one-to-one endomorphism of $M$ as a semigroup, but is
not a monoid homomorphism.
\end{lemma}

\begin{proof}
That~\eqref{d.bicyc} respects multiplication, but does
not send $1$ to $1,$ is immediate.
Moreover, if $y\,z\,x=y\,z'\,x,$ then multiplying
by $x$ on the left and $y$ on the right, we get $z=z',$
so the map is one-to-one.
\end{proof}

The corresponding observation is also easily seen to hold for rings.

\section{Adjoining inverses to not necessarily finite families
of elements}\label{S.inf}

We begin this section with a result about a situation
far more general than the
one considered in Proposition~\ref{P.oE_T} and the second paragraph of
Proposition~\ref{P.T=mnl}, but with a necessarily
much weaker conclusion:

\begin{proposition}\label{P.gen_hm}
Let $X$ be a set, and $f$ a homomorphism from $\langl X\rangl$
to an arbitrary monoid~$M.$
Then the set
\begin{equation}\begin{minipage}[c]{35pc}\label{d.N}
$N\ =\ \{w\in \langl X\rangl~|~f(w)$ is invertible in $M\}$
\end{minipage}\end{equation}
is a {\em free} submonoid of $\langl X\rangl,$
whose {\rm(}unique{\rm)} free generating
set is an overlap-free subset $T\subseteq\langl X\rangl\setminus\{1\}.$

Conversely, every overlap-free subset
$T\subseteq\langl X\rangl\setminus\{1\}$
arises in this way from such a homomorphism; e.g.,
by the first assertion of Proposition~\ref{P.T=mnl}, from
the natural map $\langl X\rangl\to\langl X : \I(T)\rangl.$
\end{proposition}

\begin{proof}
Given $f$ as above, and defining $N$ by~\eqref{d.N}, let
\begin{equation}\begin{minipage}[c]{35pc}\label{d.T=}
$T \ = \ \{w\in N\setminus\{1\}~|~w$ is not a product
of two nonidentity elements of $N\}.$
\end{minipage}\end{equation}

Without using any fact about $N$ other than that it is
a submonoid of $\langl X\rangl,$ we can see that $T$ generates $N.$
(Given $w\in N\setminus\{1\},$ assume inductively that all elements
of $N$ of shorter length than $w$ belong to $\langl T\rangl.$
Then if $w$ cannot be factored in $N,$ we have $w\in T,$
while if it can, our inductive hypothesis implies that
$w\in\langl T\rangl.)$

Now assume $N$ as in~\eqref{d.N}.
If two elements of $T$ overlap, write them $ab$ and $bc,$
with $b\neq 1$ and $a$ and $c$ not both equal to $1.$
Then the image of $b$ in $M$ is both right and left
invertible, hence invertible, hence so are the images
of $a$ and $c,$ so if $a\neq 1$ we get
a contradiction to the assumption that $ab$ belongs to~\eqref{d.T=},
and dually we get a contradiction
if $c\neq 1;$ so $T$ is indeed overlap-free.
Hence by the last assertion of Proposition~\ref{P.inv},
the submonoid $N$ of $\langl X\rangl$ generated by $T$ is free on $T.$

The final converse statement is straightforward.
\end{proof}

Note that despite the nice characterization of
the submonoids $N$ that arise as in~\eqref{d.N}, very
little can be said about $f(N)$ in such situations.
The map $f$ need not be one-to-one on $T,$ and even if it is,
it may not be one-to-one on $N,$ and even if that holds,
so that the submonoid $f(N)\subseteq M$ is free on $f(T),$
the subgroup of $M$ that it generates may not be free.
For an example of the last sort, let $f$
map the free monoid $\langl x,\,y\rangl$ into the group
$G$ of maps $t\mapsto at+b$ $(a\neq 0)$ of the real line into itself,
by sending $x$ to $t\mapsto t/2$ and $y$ to $t\mapsto (t+1)/2.$
Looking at the behavior of these maps on the unit interval $[0,\,1],$
one can show that $f$ is one-to-one; but since the commutator subgroup
of $G$ is abelian, the induced homomorphism from the free group
$\langl x,\,y : \I(x),\,\I(y)\rangl$ to $G$ cannot be one-to-one.

But in the case of Proposition~\ref{P.gen_hm} where
$M$ has the form $\langl X:\I(S)\rangl,$ we can recover essentially
everything we proved in~\S\ref{S.main} for $S$ finite, except the
assertion that $T$ can be computed algorithmically from~$S\,{:}$

\begin{proposition}\label{P.inv_gen}
Let $X$ be a set, and $S$ a subset of $\langl X\rangl.$

Then there is a unique overlap-free subset
$T\subseteq \langl X\rangl\setminus\{1\}$ which is
inverse-equivalent to~$S.$
An element $w\in\langl X\rangl\setminus\{1\}$ belongs to $T$ if and
only if it satisfies the following equivalent conditions:\vspace{.3em}

{\rm(i)} $w$ is invertible in $\langl X:\I(S)\rangl,$ but
is not a product of two elements of $\langl X\rangl\setminus\{1\}$
that are invertible in $\langl X:\I(S)\rangl.$\vspace{.3em}

{\rm(ii)} $w$ is invertible in $\langl X:\I(S)\rangl,$ but
does not have a proper left divisor in $\langl X\rangl\setminus\{1\}$
that is left invertible in $\langl X:\I(S)\rangl.$\vspace{.3em}

{\rm(ii$\!'\!$)} $w$ is invertible in $\langl X:\I(S)\rangl,$ but
does not have a proper right divisor in $\langl X\rangl\setminus\{1\}$
that is right invertible in $\langl X:\I(S)\rangl.$\vspace{.3em}

{\rm(iii)}  There exists a finite subset $S'$ of $S$ such that
for all finite $S''$ with $S'\subseteq S''\subseteq S,$
$w$ belongs to the overlap-free set $T''$ inverse-equivalent
to $S''$ {\rm(}whose existence and uniqueness were shown in
Propositions~\ref{P.oE_T} and~\ref{P.T=mnl}{\rm)}.
\end{proposition}

\begin{proof}
The first assertion of Proposition~\ref{P.gen_hm},
applied to the case $M=\langl X:\I(S)\rangl,$ shows that the elements
$w\in\langl X\rangl\setminus\{1\}$ that satisfy~(i) comprise
the unique overlap-free subset of $\langl X\rangl\setminus\{1\}$ that
is inverse-equivalent to~$S.$
It remains to show the equivalence of~(i) with the
other conditions listed.
Since~(i) and~(iii) are left-right symmetric, while~(ii$\!'\!$)
is the left-right dual of~(ii),
it suffices to show~(i), (ii) and~(iii) equivalent.

We get (i)$\implies$(ii) by contradiction:  If an element
$w\in\langl X\rangl\setminus\{1\}$
is invertible in $\langl X:\I(S)\rangl$ and has a factorization
$w=uv$ in $\langl X\rangl\setminus\{1\}$
with $u$ left invertible in $\langl X:\I(S)\rangl,$
then $u$ is both left and right invertible in the latter monoid
(right invertible because it is a left divisor of the
invertible element $w),$ hence it is invertible
there, hence $v=u^{-1}w$ is also invertible there,
so the factorization $w=uv$ contradicts~(i).

(ii)$\implies$(i) is immediate, since a failure of (i) is clearly
a failure of~(ii).

We shall prove (i) equivalent to (iii) using the
observation that $\langl X:\I(S)\rangl$ is the direct limit
of the monoids $\langl X:\I(S')\rangl$ as $S'$ ranges over
the finite subsets of $S.$

Assuming~(iii), choose $S'$ for our element $w$ as in~(iii),
and let $T'$ be the overlap-free set inverse-equivalent to $S'$
given by Proposition~\ref{P.oE_T}.
Then~$w\in T'$ is invertible
in $\langl X:\I(T')\rangl=\langl X:\I(S')\rangl,$ hence
also invertible in $\langl X:\I(S)\rangl,$ proving
the first condition of~(i).
Next, suppose by way of contradiction that $w$ could be written
as a product $w=uv$ where $u,\,v\in\langl X\rangl\setminus\{1\}$
both become invertible in $\langl X:\I(S)\rangl.$
Then taking $S'$ as above, and letting $S_u$ and $S_v$
be finite subsets of $S$ such that $u$ becomes invertible
in $\langl X:\I(S_u)\rangl$ and $v$ in $\langl X:\I(S_v)\rangl,$
we see that for $S''=S'\cup S_u \cup S_v,$ (iii) asserts
on the one hand that $w$ belongs to the overlap-free set $T'',$
and on the other hand that it
factors nontrivially into the elements $u$ and $v$
invertible in $\langl X:\I(T'')\rangl,$ a contradiction.
So~(iii)$\implies$(i).

The reverse implication is quicker: \ Assuming~(i),
$w$ is invertible in $\langl X:\I(S)\rangl$ hence in
$\langl X:\I(S')\rangl$ for some finite $S'\subseteq S;$
and if the final part of~(iii) failed, this would mean
that $w$ had a nontrivial factorization in $\langl X\rangl$
into factors invertible in some $\langl X:\I(S'')\rangl,$
hence invertible in $\langl X:\I(S)\rangl,$ contradicting~(i).
\end{proof}

We remark that though the assumption that $S$ was finite
was used in Proposition~\ref{P.oE_T}
and the second paragraph of Proposition~\ref{P.T=mnl},
we never assumed the generating
set $X$ finite; and it would not have simplified
things if we had, since for finite $X,$ $\langl X\rangl-\{1\}$
can still contain an infinite overlap-free subset~$T.$
E.g., for $X=\{x,\,y,\,z\},$ the set
$T=\{x\,y^i\,z~|~i\geq 0\}$ is overlap-free.
And in fact, if we take for $T$ the subset
of the latter set gotten by letting $i$ range over an
algorithmically undecidable subset of the natural numbers,
we get a non-algorithmically-decidable~$T.$
(With a little thought one can see that the same observations
hold with $X=\{x,\,y\}$ and elements $x\,y^i\,z$ replaced
by $x^i\,y^i.)$

In another direction, I have not investigated either of
\begin{question}\label{Q.iterate}
Suppose $X$ is a set and $S_1$ a subset of $\langl X\rangl-\{1\},$ and
we then take a subset $S_2\subseteq$ \mbox{$\langl X:\I(S_1)\rangl,$}
and adjoin to $\langl X:\I(S_1)\rangl$ universal
inverses to all elements of $S_2,$ getting a monoid
that we might denote $\langl X:\I(S_1):\I(S_2)\rangl.$

Can we get for such a monoid structure results of a similar
nature to those obtained above?
More generally, can we get such results for monoids
$\langl X:\I(S_1):\,\ldots\,:\I(S_n)\rangl$ obtained
by iterating this approach in the obvious way?
\end{question}

\begin{question}\label{Q.1-sided}
Suppose $X$ is a set, and $S_l,$ $S_r$ are subsets
of $\langl X\rangl-\{1\}.$
Let $\langl X:\I_l(S_l),\,\I_r(S_r)\rangl$
denote the monoid obtained by adjoining
a universal {\em left} inverse to each element of $S_l,$ and
a universal {\em right} inverse to each element of $S_r.$

Or more generally, since $\!1\!$-sided inverses are not
in general unique, suppose $c_l$ and $c_r$ are cardinal-valued
functions on $S_l$ and $S_r,$ and we
let $\langl X:\I_l(S_l,\,c_l),\,\I_r(S_r,\,c_r)\rangl$
be the monoid obtained by adjoining to each $w\in S_l$
a universal $\!c_l(w)\!$-tuple of left inverses, and to each
$w\in S_r$ a universal $\!c_r(w)\!$-tuple of right inverses.

Can one obtain nice structure results for such monoids?
{\rm(}And perhaps for monoids given
by iterated versions of these constructions, analogous
to those asked about in Question~\ref{Q.iterate} above?{\rm)}
\end{question}

\section{Comparison with special monoids}\label{S.special}

There has been considerable research on what are called
``special monoids'':~ monoids presented by a finite set of generators
and a finite family of relations of the form $w_i=1.$
A striking result from~\cite{Adjan}, summarized in
\cite[p.\,1876, statement~(A1), and
Theorem~1.1 and following paragraph]{RG+NR},
is that any special monoid presented by just {\em one} such
relation has solvable word problem.

The solvability of the word problem in any monoid
of the form $\langl X:\I(S)\rangl$ with $X$ and $S$ finite,
shown in~\S\ref{S.main} above, can
also be deduced from that result of~\cite{Adjan}: \ If we
form as in Lemma~\ref{L.1-elt} above a word $w\in\langl X\rangl$
such that $\langl X:\I(w)\rangl\cong\langl X:\I(S)\rangl,$
and then adjoin one more generator $z$ to $X,$ I claim
that the above monoid
can also be presented as $\langl X\cup\{z\}\ |\ w\,z\,w=1\rangl.$
Indeed, the relation $w\,z\,w=1$ makes $w$ a left
and right factor of an invertible element, hence invertible, hence
$z=w^{-2},$ so $z$ can be dropped from the
generators if we include in the presentation an inverse $\I(w).$
So facts about one-relator special monoids apply in particular
to monoids of the sort studied in~\S\ref{S.main} above.

On the other hand, special monoids defined using more than one
relation $w_i=1$ are much less well-behaved.
Indeed, it is proved in~\cite{PN+CO+FO} that whether such
a monoid is a group is undecidable.
Moreover, if such a monoid {\em is} a group, it can
have undecidable word problem, since
any finitely presented group can be finitely presented as a
special monoid, and there are known to be finitely presented
groups $G$ with undecidable word problem.

In any case, the undecidability result
of~\cite{PN+CO+FO} answers a question I was wondering about.
Suppose $X$ is a set and $\{w_1,\,\dots,\,w_n\}$ a finite subset
of $\langl X\rangl.$
In view of the natural homomorphism
$\langl X:\I(w_1),\,\dots,\,\I(w_n)\rangl\to
\langl X\,|\,w_1=1,\,\dots,\,w_n=1\rangl,$ see that
\begin{equation}\begin{minipage}[c]{35pc}\label{d.inclusion}
The set of elements of $\langl X\rangl$ that become invertible
in $\langl X:\I(w_1),\,\dots,\,\I(w_n)\rangl$
is contained in the set that become invertible in
$\langl X\,|\,w_1=1,\,\dots,\,w_n=1\rangl.$
\end{minipage}\end{equation}
Is that inclusion always an equality?
The answer must be no, since we can see
from the results of~\S\ref{S.main} that it is algorithmically
decidable whether all elements of $\langl X\rangl$
become invertible in $\langl X:\I(w_1),\,\dots,\,\I(w_n)\rangl;$
hence if the inclusion of~\eqref{d.inclusion}
were an equality, it would be
decidable whether all elements of $\langl X\rangl$
become invertible in $\langl X\,|\,w_1=1,\,\dots,\,w_n=1\rangl,$
i.e., whether that monoid is a group,
contradicting the result of~\cite{PN+CO+FO} mentioned.

Some of the many works on special monoids are \cite{Adjan},
\cite{PN+CO+FO}, \cite{FO+LZ}, \cite{DP+PS}, \cite{LZ}.
In particular,~\cite[\S5]{LZ} gives simplified proofs of
results in the earlier literature, such as that the word problem
for a finitely presented special monoid is decidable if and only if
the word problem for its group of units is decidable.
There is also much information about special monoids
in \cite{RG+NR} and \cite{GL},
though the main subject of the former paper is {\em inverse monoids}
(monoids $M$ such that for every $x\in M$ there is a unique
$x'\in M$ such that $x\,x'\,x=x$ and $x'\,x\,x'=x',$
and where for all $x$
and $y,$ the elements $x\,x'$ and $y\,y'$ commute), and the main
subject of the latter paper is monoids presented by an
arbitrary single relation~$w_1=w_2.$

\section{Acknowledgements}\label{S.ackn}

I am indebted to Benjamin Steinberg for informing me that
it is known to be undecidable whether a finitely
presented special monoid is a group, and to Robert Gray for further
information on what is known about such monoids, and
for helpful comments on earlier drafts of this note.

\end{document}